\documentclass[a4paper]{article}
\usepackage{amsfonts,amsmath,amstext}
\usepackage{graphicx}

\def\rh#1{{\rm\bf H}_{\mathbb R}^{#1}}
\def\ch#1{{\rm\bf H}_{\mathbb C}^{#1}}
\def\pu#1{{\rm PU}(#1,1)}

\def\u#1{{\rm U}(#1,1)}
\def\isom#1{{\rm Isom}(#1)}
\def\Heis{{\mathcal H}}
\def\Siegel{{\mathcal S}}

\def\mod#1{\vert #1\vert}

\def\hherm#1{{\langle\!\langle #1 \rangle\!\rangle}}

\def\R{\mathbb R}
\def\C{\mathbb C}

\def\P{\mathbb P}

\def\herm#1{\langle #1\rangle}
\newcommand{\dps}{\displaystyle}
\def\qi{q_\infty}

\newenvironment{proof}{\noindent\normalsize {\sc Proof}:}{{\hfill \rule{1mm}{3mm}}}

\newtheorem{theorem}{Theorem}[section]
\newtheorem{co}{Corollary}[section]
\newtheorem{prop}{Proposition}[section]

\title{On the moduli space of quadruples of points in the boundary of complex hyperbolic space}
\author{Heleno Cunha and Nikolay Gusevskii~\thanks{Corresponding author. Supported by CNPq and FAPEMIG.}\\
{\small Departamento de Matem\'{a}tica}\\
{\small Universidade Federal de Minas Gerais}\\
{\small 30123-970  Belo Horizonte - MG - Brazil}}
\date{}

\begin{document}
\maketitle


\begin{abstract}
\noindent We consider the space $\mathcal M$ of ordered quadruples
of distinct points in the boundary of complex hyperbolic $n$-space,
$\ch{n},$ up to its holomorphic isometry group ${\rm PU}(n,1).$ One
of the important problems in complex hyperbolic geometry is to
construct and describe a moduli space for $\mathcal M$. For $n=2$,
this problem was considered by Falbel, Parker, and Platis. The main
purpose of this paper is to construct a moduli space for $\mathcal M
$ for any dimension $n \geq 1$. The major innovation in our paper is
the use of the Gram matrix instead of a standard position of points.
\end{abstract}

\quad {\sl MSC:} 32H20; 20H10; 22E40; 57S30; 32G07; 32C16

\quad {\sl Keywords:} Complex hyperbolic space; Cross-ratio;
Cartan's invariant; Gram matrix.

\section*{Introduction}

An important problem in complex hyperbolic geometry is to classify
ordered $m$-tuples of distinct points in complex hyperbolic
$n$-space, $\ch{n}$, or in its boundary,  $\partial\ch{n}$,  up to
congruence in the holomorphic isometry group  ${\rm PU}(n,1)$ of
$\ch{n}$. This problem is trivial for $m=1,2$, and it was completely
solved by Cartan and Brehm for triangles. If the vertices of a
triangle are in the boundary of complex hyperbolic space, then
Cartan's angular invariant $\mathbb A$ associated to this triangle
describes its congruence class in ${\rm PU}(n,1)$, Cartan \cite{C},
see also Goldman \cite{Go}. So, we can describe the moduli space of
triangles in the boundary of complex hyperbolic space as the closed
interval of real numbers $[-\pi/2, \pi/2]$. If the vertices of a
triangle are in $\ch{n}$, then its congruence class in ${\rm
PU}(n,1)$ is described by the side lengths and Brehm's shape
invariant, see Brehm \cite{Br}. If $m>3$, the problem becomes more
difficult. It was shown by Brehm and Et-Taoui \cite{BrEt} that for
points in $\ch{n}$, the congruence class of an $m$-tuple in ${\rm
PU}(n,1)$, $m>3$, is determined by the congruence classes of all of
its triangles. Easy examples show that this is not true for
$m$-tuples in the boundary of complex hyperbolic space. This implies
that Cartan's invariants of all of the triangles of an ordered
$m$-tuple in the boundary of complex hyperbolic space in the case
$m>3$ and $n>1$ do not determine its congruence class in ${\rm
PU}(n,1)$, and, therefore, we need something else to construct a
moduli space of such configurations. For instance, one can try to
use some invariants of four points.

\vspace{2mm}

For $n=2$ and $m=4$, this problem was considered by Falbel \cite{F},
see also Falbel-Platis \cite{FP}, and Parker-Platis \cite{PP1},
\cite{PP2}. For the convenience of the reader, we describe briefly
their results. To do it, we recall the definition of the
Kor\'{a}nyi-Reimann complex cross-ratio.

In \cite{KR}, the classical cross-ratio was generalized by
Kor\'{a}nyi and Reimann to complex hyperbolic space, see also
Goldman \cite{Go}. Let $p=(p_1,p_2,p_3,p_4)$ be an ordered quadruple
of distinct points of $\partial\ch{n}$. Following Goldman \cite{Go},
we define the Kor\'{a}nyi-Reimann complex  cross-ratio, or simply
the complex cross-ratio, of $p$, as follows

\begin{equation*}
X(p)=X(p_1,p_2,p_3,p_4)= \dps \frac{\herm{P_3,P_1}\herm{P_4,P_2}}{\herm{P_4,P_1}\herm{P_3,P_2}},
\end{equation*}
\noindent where $\herm{X,Y}$ is the Hermitian product in
$\mathbb{C}^{n+1}$ of signature $(n,1)$, and $P_i \in
\mathbb{C}^{n+1}$ is a null lift of $p_i$, see \cite{Go}. This
number is independent of the chosen lifts $P_i$. Moreover, $X(p)$ is
invariant with respect to the diagonal action of ${\rm PU}(n,1)$.
Basic properties of the complex cross-ratio may be found in Goldman
\cite{Go}.

\vspace{2mm}

Let $\mathcal M $ be the configuration space of ordered quadruples
of distinct points in the boundary of the complex hyperbolic plane,
that is, the quotient of the set of ordered quadruples of distinct
points of $\partial\ch{2}$  with respect to the diagonal action of
${\rm PU}(2,1)$ equipped with the quotient topology.

\vspace{2mm}

Falbel's cross-ratio variety $\mathcal X$ is the subset of
$\mathbb{C}^3$ defined by the following equations:

\begin{itemize}
\item $|\omega_0| |\omega_1| |\omega_2|-1=0,$
\item $|\omega_0 -1|^2 -1 + |\omega_0|^2(|\omega_1-1|^2-1)+
|\omega_0|^2 |\omega_1|^2(|\omega _2-1|^2-1)=0,$
\end{itemize}
see \cite{F}, \cite{FP}. It is easy to see that $\mathcal X$ is a real algebraic
variety of dimension four in $\mathbb{C}^3$.

\vspace{2mm}

Let $m(p)\in \mathcal M $ be the point represented by
$p=(p_1,p_2,p_3,p_4)$. Falbel \cite{F}, \cite{FP} defines the map
$\pi:\mathcal M\longrightarrow \mathcal X$ by the following formula
\begin{equation*}
\pi:m(p)\rightarrow (\omega_0, \omega_1, \omega_2)=
\bigr(X(p_1,p_2,p_3,p_4),\
X(p_1,p_4,p_2,p_3),X(p_1,p_3,p_4,p_2)\bigl).
\end{equation*}
\noindent The main result of \cite{F} is that this map is a
bijection, see Proposition 2.4 there. In other words, this means
that the ${\rm PU}(2,1)$-congruence class of ordered quadruples of
distinct points in the boundary of the complex hyperbolic plane is
completely determined by these three complex cross-ratios satisfying
these two real equations, that is, Falbel's cross-ratio variety
could serve as a moduli space for the configuration space $\mathcal
M$.

\vspace{2mm}

In \cite{PP1}, Parker and Platis used a slightly different
construction from that of Falbel to describe $\mathcal M$. They
define the map $\pi': \mathcal M \longrightarrow  \mathbb  C^{3}$ by
the following formula
\begin{equation*}
\pi':m(p)\rightarrow \bigl (X_1=X(p_1,p_2,p_3,p_4), X_2= X(p_1,p_3,p_2,p_4), X_3=X(p_2,p_3,p_1,p_4)\bigr).
\end{equation*}
They proved that $X_1, X_2, X_3$ satisfy the following relations

\begin{itemize}
\item $| X_2| =| X_1| | X_3|,$
\item $2| X_1| ^2\rm{Re}(X_3)=| X_1| ^2+| X_2| ^2+1-2\textrm{Re}(X_1+X_2).$
\end{itemize}
The Parker-Platis cross-ratio variety $\mathcal X'$ is the subset of
$\mathbb{C}^3$, where these relations are satisfied. The main result
of \cite{PP1} is that the map $\pi': \mathcal M \longrightarrow
\mathcal X'$ defined above is a bijection, see Proposition 5.5 and
Proposition 5.10 in \cite{PP1}. Again, this implies that the
Parker-Platis cross-ratio variety could be considered as a moduli
space for the configuration space $\mathcal M$. It is easy to see
that $\mathcal X$ is homeomorphic to $\mathcal X'$, \cite{PP1},
\cite{PP2}.

\vspace{2mm}

These results of Falbel and Parker-Platis have been used in a series
of recent papers, see, \cite{FP}, \cite{PP1}, \cite{W}, and others.
For instance, in \cite{PP1}, Parker and Platis used points of
$\mathcal X'$ in their generalization of Fenchel-Nielsen coordinates
to the complex hyperbolic setting, see also Will, \cite{W} for
related topics. Falbel and Platis \cite{FP} used Falbel's
cross-ratio variety to describe some geometric properties of the
configuration space $\mathcal M$.

\vspace{2mm}

In the present article, we show that, unfortunately, these results
of Falbel and Parker-Platis on the description of the configuration
space $\mathcal M,$ in spite of their important contribution to the
solution to this problem, are not correct.

The reason is that the maps $\pi:\mathcal M \longrightarrow \mathcal
X$ and $\pi':\mathcal M\longrightarrow \mathcal X'$ are not
injective. The reader can find explicit examples that show this for
the Parker-Platis variety in Section 3.3 of the paper. In fact, we
show that there exists a $1$-parameter family of pairs of distinct
points in $\mathcal M$, $m(t)$ and $m^{*}(t)$, corresponding to
$\mathbb{C}$-plane quadruples such that $\pi' (m(t))=\pi'
(m^{*}(t))$ (a quadruple of points in $\partial\ch{n}$ is
$\mathbb{C}$-plane if all its points lie in the boundary of a
complex geodesic, or equivalently, in a $\mathbb{C}$-circle).  We
will also prove that the above examples describe all the points of
$\mathcal M,$ where the map $\pi':\mathcal M \longrightarrow
\mathcal X'$ is not injective. Of course, the same examples serve
for Falbel's variety.

\vspace{2mm}

All this implies that both Falbel's and the Parker-Platis
cross-ratio varieties cannot serve as moduli spaces for the
configuration space $\mathcal M$.

\vspace{2mm}

We now state our results. The main purpose of our paper is to
present a correct description of the configuration space $\mathcal
M$. Also, we construct a moduli space of ordered quadruples of
distinct points in the boundary of complex hyperbolic space $\ch{n}$
for any dimension $n \geq 1$. The major innovation in our paper is
the use of the Gram matrix instead of a standard position of points
used in \cite{F}, \cite{FP}, \cite{PP1}.

\vspace{2mm}

Let $p=(p_1,p_2,p_3,p_4)$ be an ordered quadruple of distinct points
of $\partial\ch{2}$, and let $m(p)\in \mathcal M $ be the point
represented by $p=(p_1,p_2,p_3,p_4)$.  We define the map
$$\tau : \mathcal M \longrightarrow  \mathbb C^{2} \times \mathbb R$$
by the following formula:
\begin{equation*}
\dps \tau :m(p)\mapsto  \bigl (X_1=X(p_1,p_2,p_3,p_4),\ X_2= X(p_1,p_3,p_2,p_4),\
\mathbb A=
\mathbb{A} (p_1,p_2,p_3)\bigr),
\end{equation*}
where  $\mathbb A= \mathbb{A} (p_1,p_2,p_3)$ is the Cartan invariant of the
triple $(p_1,p_2,p_3)$.

\vspace{2mm}

\noindent {\bf Theorem A} \ {\sl Let $X_1, X_2, \mathbb A$ be the
numbers defined by $\tau$. Then they satisfy the following relation:
\begin{equation*}
\dps -2{\rm Re}(X_1+X_2) -2{\rm Re} (X_1\ \overline{X}_2\ e^{-
i2\mathbb A})+ \mod {X_1}^2 +\mod {X_2}^2 +1=0.
\end{equation*}}

Our main result is the following theorem.

\vspace{2mm}

\noindent {\bf Theorem B} \ {\sl The configuration space $\mathcal M
$ is homeomorphic to the set of points  $X=(X_1, X_2, A)$ in
$\mathbb C_{*}^{2} \times \mathbb R$ defined by
\begin{equation*}
\dps -2{\rm Re}(X_1+X_2) -2{\rm Re} (X_1\ \overline{X}_2\ e^{- i2
A})+ \mod {X_1}^2 +\mod {X_2}^2 +1=0,
\end{equation*}
subject to the following restrictions:

 $$\dps -\pi /2 \leq A \leq \pi /2,\ \ \ {\rm Re} (X_1 \ e^{-i A})\geq
 0,$$} where $\mathbb {C}_{*}=\mathbb {C}\setminus \{0\}.$

\vspace{2mm}

We denote this set by $\mathbb M$ and call $\mathbb M$ the moduli
space for $\mathcal M$. The map $ \tau : \mathcal M \longrightarrow
\mathbb M$ is a homeomorphism provided that $\mathbb M$ is equipped
with the topology induced from $\mathbb C^{2} \times \mathbb R$.

\vspace{2mm}

Finally, we get a generalization of Theorem B for any dimension. Let
$\mathcal M (n,4)$ be the configuration space of ordered quadruples
of distinct points in the boundary of complex hyperbolic space of
dimension $n\geq 2.$

\vspace{2mm}

\noindent {\bf Theorem C} \ {\sl $\mathcal {M}(n,4)$  is
homeomorphic to the set of points $X=(X_1, X_2, A) \in \mathbb
C_{*}^{2} \times \mathbb R$ defined by
\begin{equation*}
\dps -2{\rm Re}(X_1+X_2) -2{\rm Re} (X_1\ \overline{X}_2\ e^{-\ i2
A})+ \mod {X_1}^2 +\mod {X_2}^2 +1\leq 0,
\end{equation*}
$$\dps -\pi /2 \leq  A \leq \pi /2,\ \ \ {\rm Re} (X_1 \ e^{-iA})\geq
0.$$ The equality in the first inequality happens if and only if the
quadruples are in the boundary of a complex hyperbolic $2$-space.}

\vspace{2mm}

It should be noticed here that the problem of the construction of
the moduli space for the configuration space of ordered $m$-tuples
of distinct points in $\partial\ch{n}$, in the case $m=n+1$, was
considered by Hakim-Sandler, see \cite{HS}. The invariants they used
are too complicate to be described here.

The paper is organized as follows. In Section 1, we review some
basic facts in complex hyperbolic geometry. In Section 2, we obtain
the principal formulae we need. Section 3.1 is devoted to the
construction of the moduli space for the configuration space of
ordered quadruples of  points in the boundary of the complex
hyperbolic plane. In Section 3.2, we describe some interesting
subsets of the moduli space and describe its topology. In Section
3.3, we discuss a relation between the configuration space and the
Falbel-Parker-Platis cross-ratio varieties. Finally, in Section 3.4,
we describe the moduli space for the configuration space of ordered
quadruples of points in the boundary of complex hyperbolic space of
any dimension.

\section{Complex hyperbolic space and its boundary}

Let $V^{n,1}$ be a $(n+1)$-dimensional $\mathbb C$-vector space
equipped with a Hermitian form $\herm{-,-}$ of signature of $(n,1)$.
The projective model of {\sl complex hyperbolic space} $\ch{n}$ is
the set of negative points in the projective space $\P\C^{n}.$ It is
well known that $\ch{n}$ can be identified with the unit open ball
in $\mathbb C^n.$ We will consider $\ch{n}$ equipped with the
Bergman metric, see \cite{Go}. Then $\ch{n}$ is a complete K\"ahler
manifold of constant holomorphic sectional curvature $-1$. The
boundary of $\ch{n},$ denoted by $\partial {\ch{n}}$, is the
$(2n-1)$-sphere formed by all isotropic points. Let $\u{n}$ be the
unitary group corresponding to this Hermitian form. The holomorphic
isometry group of $\ch{n}$ is the projective unitary group $\pu{n}$,
and the full isometry group $\isom{\ch{n}}$ is generated by $\pu{n}$
and complex conjugation.

For the purposes of our paper it is convenient to work with
coordinates in $V^{n,1}$ in which the Hermitian product is
represented by:
$$\herm{Z,W}\ =\ z_1\overline{w}_{n+1}\ +z_2\overline{w}_2+\ \cdots \
+\ z_n\overline{w}_n\ + \ z_{n+1}\overline{w}_1\ , $$
where
$$Z\ =\ \begin{bmatrix} z_1 \\ \vdots \\ z_n \\ z_{n+1} \end{bmatrix}
\ \ \ \ \ {\rm and}\ \ \ \ \ W\ =\ \begin{bmatrix} w_1 \\ \vdots \\
w_n \\ w_{n+1} \end{bmatrix}\ . $$

These coordinates  were used by Burns-Shnider \cite{BS}, Epstein
\cite{E} (the "second Hermitian form"), Parker-Platis \cite{PP1},
and also in \cite{GuP1}, \cite{GuP2}, \cite{DG}, among others. In
what follows, we denote by $\mathbb {C}^{n,1}$ the Hermitian vector
space $V^{n,1}$ equipped with these coordinates. In these
coordinates, the boundary of complex hyperbolic space is identified
with the one point compactification of the boundary of the Siegel
domain $\Siegel^{n}$. Following Goldman-Parker \cite{GoP}, we give
the Siegel domain {\sl horospherical coordinates.} Horospherical
coordinates are defined as follows: The height $u \in \mathbb
{R}_{+}$ of a point $q=\mathbb {P}(Z)$ in $\Siegel^{n}$ is defined
by $u=-\herm{Z,Z}/2.$ The locus of points where the height is
constant is called a {\sl horosphere} and naturally carries the
structure of the Heisenberg group $\Heis=\{(z,t):z \in
\mathbb{C}^{n-1}, t \in \mathbb R\}=\mathbb{C}^{n-1}\times \mathbb
R.$ The horospherical coordinates of a point in the Siegel domain
are just $(z,t,u) \in \mathcal H \times \mathbb{R}_{+}.$ We define
the relation between horospherical coordinates on the Siegel domain
and vectors in $\mathbb {C}^{n,1}$ by the following map:

\begin{equation*}  \label{horo1}
\psi:(z,t,u)\, \longmapsto\, \dps\begin{bmatrix}
\dps -\hherm{z,z} - u + it  \\
\dps z\sqrt{2}\\\ \dps 1 \end{bmatrix} \quad\hbox{ for } (z,t,u)\in
\overline{\Siegel}-\{q_\infty\};\qquad\psi:q_\infty\longmapsto
\begin{bmatrix} 1 \\ 0 \\ \vdots \\ 0 \end{bmatrix},
\end{equation*}

\noindent where $\hherm{z,z'}$ is the standard Hermitian product on
$\C^{n-1},$ and $q_\infty =\infty$ is a distinguished point in the
boundary of $\ch{n}$. We call the vectors $\psi(z,t,u)$ and $\psi
(q_{\infty})$ the {\sl standard lifts }of $(z,t,u)$ and
$q_{\infty}$.

This shows that the boundary of $\ch{n},$ \ $\partial {\ch{n}}$, \
may be thought of as the one point compactification of the
horosphere of height $u=0$.

\medskip
There are two types of totally geodesic submanifolds of $\ch{n}$ of
real dimension two:
\begin{itemize}
\item {\sl Complex geodesics} (copies of $\ch{1}$) have constant
sectional curvature $-1$.
\item {\sl Totally real geodesic 2-planes} (copies of $\rh{2}$)
have constant sectional curvature $-1/4$.
\end{itemize}
\medskip

The intersection of a complex geodesic $L$ with $\partial\ch{n}$ is
called a {\sl chain} $C = \partial L$. Chains passing through $\qi$
are called {\sl vertical} or {\sl infinite}. A chain which does not
contain $\qi$ is called {\sl finite}. If $L$ is a complex geodesic,
then there is a unique inversion in $\pu{n}$ whose fixed-point set
is $L$. It acts on $\partial\ch{n}$ fixing $C = \partial L$.

\medskip
The intersection of a totally real geodesic 2-subspace with
$\partial\ch{n}$ is called an $\R$-circle.  Just as for chains, an
$\R$-circle in $\Heis$ is one of two types, depending on whether or
not it passes through $\qi$. An $\R$-circle is called {\sl infinite}
if it contains $q_\infty$, otherwise, it is called {\sl finite}. If
$R$ is a totally real geodesic 2-subspace, then there is a unique
inversion (anti-holomorphic automorphism of $\ch{n}$) whose
fixed-point set is $R$. It acts on $\partial\ch{n}$ fixing $
\partial R$.
\bigskip

\section{Gram matrix, Cartan's invariant, Cross-ratio}

\subsection{The Gram matrix}

Let $p=(p_1, \cdots ,p_m)$ be an ordered $m$-tuple of distinct
points in $\partial {\ch{n}}.$  Then we consider the Hermitian
$m\times m$-matrix
$$G= \ G(p) \ = \ (g_{ij}) \ = \ (\herm{P_i,P_j}),$$
where $P_i$ is a lift of $p_i$. We call $G$ a {\sl Gram matrix}
associated to the $m$-tuple $p$. Of course, $G$ depends on the
chosen lifts $P_i$. When replacing $P_i$ by $\lambda_i P_i$, where
$\lambda_i \in \mathbb {C}_{*}=\mathbb {C}\setminus \{0\},$ we get
$\tilde{G}=D^{*}\ G D,$ where $D$ is the diagonal matrix,
$D=(\lambda_i \ \delta _{ij})$, with $\delta _{ij}=0$, if $i\neq j$,
and $\delta _{ij}=1$, if $i=j.$

We say that two  Hermitian $m \times m$ - matrices $H$ and
$\tilde{H}$ are {\sl equivalent} if there exists a diagonal matrix
$D$, $D=(\lambda_i \ \delta _{ij})$, $\lambda_i \in \mathbb
{C}_{*}$, such that $\tilde{H}=D^{*} \ H \ D.$

Thus, to each $m$-tuple $p$ of points in $
\partial {\ch{n}}$ is associated an equivalence class of Hermitian
$m\times m$ - matrices with $0's$ on the diagonal. We remark that
for any two Gram matrices $G$ and $\tilde{G}$  associated to an
$m$-tuple $p$ the equality $\det \tilde{G} = \lambda \det G$ holds,
where $\lambda
> 0$. This implies that the sign of $\det G$ does not depend on the
chosen lifts $P_i$.

\vspace{2mm}

Now we consider the case $m=4$.

\begin{prop}
Let $p=(p_1,p_2,p_3,p_4)$ be an ordered quadruple of distinct points
of $\partial\ch{n}$. Then the equivalence class of Gram matrices
associated to $p$ contains a unique matrix $G=(g_{ij})$ with
$g_{ii}=0$, $g_{12}=g_{23}=g_{34}=1$, $\mod {g_{13}}=1$.
\end{prop}
\begin{proof}
Since $p_i$ are null points, we have that $g_{ij}\neq 0$ for $i\neq
j.$ Then it is easy to see that by appropriate re-scaling we may
assume that
$$g_{12}=g_{23}=g_{34}=1.$$
Let $a=1/\sqrt{\mod{g_{13}}}$. Taking  $P'_1 =a P_1,$ $P'_2 =(1/a)
P_2,$ $P'_3 =a P_3,$  $P'_4 =(1/a) P_4$, we get the result.
\end{proof}

\vspace{2mm} We call such a matrix $G$ a {\sl normal form} of the
associated Gram matrix. Also, we call $G$ the {\sl normalized} Gram
matrix.

\vspace{2mm}

\noindent {\bf Remark} The normalization of the Gram matrices is a
common trick in complex hyperbolic geometry, see, for instance,
Brehm \cite{Br}, Brehm and Et-Taoui \cite{BrEt}, Goldman \cite{Go},
Grossi \cite{Gr}. We use the normalization considered in \cite{Gr}.

\vspace{2mm}

\begin{prop} Let $p=(p_1,p_2,p_3,p_4)$ and $p'=(p'_1,p'_2,p'_3,p'_4)$
be two ordered quadruple of distinct points of $\partial\ch{n}$.
Then $p$ and $p'$ are congruent in ${\rm PU}(n,1)$ if and only if
their associated Gram matrices are equivalent.
\end{prop}
\begin{proof}
Let us assume that $p$ and $p'$ have equivalent associated Gram
matrices.  Let $P=(P_1,P_2,P_3,P_4)$ and $P'=(P'_1,P'_2,P'_3,P'_4)$
be the quadruples of vectors in $V^{n,1}$ which represent $p_i$ and
$p_{i}'$. Let $V$ and $V'$ be the subspaces spanned by $P$ and $P'$.
Since all the points $p_i$  are distinct and isotropic it is easy to
see that the linear map $V \longrightarrow V' \subseteq V^{n,1}$
defined by $P \longmapsto P'$ is an injective isometry. Then the
result follows by applying the Witt theorem \cite{Sch}.
\end{proof}

\vspace{2mm}

\noindent {\bf Remark} This proposition follows also from Lemma
2.2.2 in \cite{Gr}.

\vspace{2mm}

\begin{co} Let $p$ and $p'$ be two ordered quadruples of distinct points
of $\partial\ch{n}$, and let $G$ and $G'$ be their normalized Gram
matrices. Then $p$ and $p'$ are congruent in ${\rm PU}(n,1)$ if and
only if $G=G'$.
\end{co}

\begin{co}
Let $p$ and $p'$ be two ordered quadruples of distinct points of
$\partial\ch{n}$. Then $p$ and $p'$ are congruent with respect to an
anti-holomorphic isometry of $\ch{n}$ if and only if their
normalized Gram matrices are conjugate.
\end{co}

\vspace{2mm}

By direct computation we obtain the following formulae.
\begin{prop}
Let $p=(p_1,p_2,p_3,p_4)$ be an ordered quadruple of distinct points
of $\partial\ch{n}$ and $G=(g_{ij})$ be the normalized Gram matrix
of $p$. Then
$$\det G= -2{\rm Re}({g}_{14})  -2{\rm Re}({g}_{13} \bar{{g}}_{24})-
2{\rm Re}({g}_{13} \bar{{g}}_{14} g_{24}) + \mod{{g}_{14}}^2 +
\mod{{g}_{24}}^2 +1.$$
\end{prop}
\begin{prop}
Let $G(i,j,k)$ be the submatrix of $G$ corresponding to the triple
$(p_i,p_j,p_k)$. Then
$$ \ \ \ \ \ \det G(1,2,3) = 2 {\rm Re}(\bar{{g}}_{13}),\ \ \ \ \ \ \  \
\ \det G(1,2,4) = 2 {\rm Re}({g}_{24} \ \bar{{g}}_{14}),$$
$$ \ \det G(1,3,4) = 2 {\rm Re}({g}_{13}  \ \bar{{g}}_{14}),  \ \
\ \  \det G(2,3,4) = 2 {\rm Re}(\bar{{g}}_{24}).$$
\end{prop}

\noindent {\bf Remark} It follows from Sylvester's Criterion that
all these determinants are negative or vanish. Also, it is easy to
see that $\det G =0$ if and only if $p$ is in the boundary of a
complex hyperbolic $2$-space and $\det G(i,j,k)=0$ if and only if
the corresponding triple lies on a chain.

\subsection{Cartan's angular invariant}
Let $p=(p_1,p_2,p_3)$ be an ordered triple of distinct points in the
boundary $\partial\ch{n}$ of complex hyperbolic $n$-space. Then {\sl
Cartan's angular invariant} $\mathbb{A}(p)$ of $p$ is defined to be
$$\mathbb{A}(p)=\arg (-\langle P_1, P_2, P_3\rangle),$$
where $P_i \in V^{n,1}$ are corresponding lifts of $p_i,$ and
$$\langle P_1, P_2, P_3\rangle=\herm{P_1,P_2} \herm{P_2,P_3}
\herm{P_3,P_1} \in \mathbb C$$ is the {\sl Hermitian triple
product.} It is verified that $\mathbb{A}(p)$ is independent of the
chosen lifts and satisfies
$$-\pi/2 \leq \mathbb{A}(p) \leq \pi/2.$$
The Cartan invariant is the only invariant of a triple of points:
$p$ and $p'$ are congruent in ${\rm PU}(n,1)$ if and only if
$\mathbb{A}(p) = \mathbb{A}(p').$ Basic properties of the Cartan
invariant can be found in Goldman \cite{Go}.

The following proposition shows a relation between Cartan's
invariant and the Gram determinants in Proposition 2.4.
\begin{prop} Let $p=(p_1,p_2,p_3,p_4)$ be an ordered quadruple of
distinct points of $\partial\ch{n}$ and $G=(g_{ij})$ be the
normalized Gram matrix of $p$. Then

$$ \ \ \ \mathbb{A}(p_1,p_2,p_3)=\arg(-\bar{g}_{13}),\ \ \  \ \  \ \
\ \ \mathbb{A}(p_1,p_2,p_4)=\arg(-g_{24}
\bar{g}_{14}),$$
$$ \mathbb{A}(p_1,p_3,p_4)=\arg(-g_{13} \bar{g}_{14}),  \ \ \ \
\  \mathbb{A}(p_2,p_3,p_4)=\arg(-\bar{g}_{24}).$$
\end{prop}
One may use this to see again that the Cartan invariant lies in the
interval $[-\pi/2,\pi/2]$.

\subsection{The complex cross-ratio}
In \cite{KR}, Kor\'{a}nyi and Reimann defined a complex-valued
invariant associated to an ordered quadruple of distinct points of
$\partial\ch{n}.$ This invariant generalizes the usual cross-ratio
of a quadruple of complex numbers. Let $p=(p_1,p_2,p_3,p_4)$ be an
ordered quadruple of distinct points of $\partial\ch{n}$. Following
Goldman, \cite{Go}, the {\sl Kor\'{a}nyi - Reimann complex
cross-ratio} is defined to be
$$ X=X(p) = \frac{\herm{P_3,P_1} \herm{P_4,P_2}}{\herm{P_4,P_1}
\herm{P_3,P_2}},$$ where $P_i \in V^{n,1}$ are corresponding lifts
of $p_i.$ It is verified that the  complex cross-ratio is
independent of the chosen lifts $P_i$ and is invariant with respect
to the diagonal action of ${\rm PU}(n,1)$. Since the points $p_i$
are distinct, $X$ is finite and non-zero. More properties of the
complex cross-ratio may be found in Goldman \cite{Go}.

\vspace{2mm}

Let $p=(p_1,p_2,p_3,p_4)$ be an ordered quadruple of distinct points
of $\partial\ch{n}$. We define

\begin{equation*}
X_1=X(p_1,p_2,p_3,p_4), \ \ X_2=X(p_1,p_3,p_2,p_4),
 \ \ X_3=X(p_2,p_3,p_1,p_4).
\end{equation*}

\vspace{2mm}

\noindent {\bf Remark} These complex cross-ratios were considered by
Parker-Platis \cite{PP1}.

\vspace{3mm}

Easy computations give the following.
\begin{prop} Let $p=(p_1,p_2,p_3,p_4)$ be an ordered quadruple of distinct
points of $\partial\ch{n}$ and $G=(g_{ij})$ be the normalized Gram
matrix of $p$. Then
$$ X_1 = \frac{\bar{g}_{13}\bar{g}_{24}}{\bar{g}_{14}}, \ \
X_2 = \frac{1}{\bar{g}_{14}}, \ \ X_3
=\frac{1}{g_{13}\bar{g}_{24}},$$
and
$$ \ \ \ \ \ g_{13}=-e^{-i\mathbb A}, \ \ \  \ g_{14}=\frac{1}{\overline{X}_2}, \ \ \
g_{24}=-\frac{\overline{X}_1}{\overline{X}_2} \  \ e^{i\mathbb A},$$
where $\mathbb A$ is the Cartan invariant of the triple
$(p_1,p_2,p_3).$
\end{prop}
Using these formulae and applying Corollary 2.1, we get the
following important result.
\begin{co} The numbers $ X_1,$ $ X_2,$ and $\mathbb A$ define
uniquely the congruence class of $p$ in ${\rm PU}(n,1)$.
\end{co}
\noindent {\bf Remark} It is easy to see that it is impossible to
express uniquely $g_{ij}$ in terms of $ X_1,$ $ X_2,$ $ X_3.$ In
fact, we have that
$$(g_{13})^2 \ =\frac{X_2}{X_1 \ X_3}, \ \  \  \  (g_{24})^2 \ =\frac{\overline{X}_1}{\
\overline{X}_2 \overline{X}_3}.$$ Therefore, the congruence class of
$p$ in ${\rm PU}(n,1)$ is not defined uniquely by $ X_1,$ $ X_2,$ $
X_3.$

\vspace{2mm}

Substituting $g_{ij}$ from the expressions in Proposition 2.6 into
the formulae in Proposition 2.3 and Proposition 2.4 and rearranging,
we get the formulae for the Gram determinants in terms of Cartan's
invariant and the complex cross-ratios.
\begin{prop} The determinant of the normalized Gram matrix of $p$ is given by
$$\det G = \frac{1}{\mod{X_2}^2} \ \ [-2{\rm Re}(X_1+X_2)
-2{\rm Re}(X_1 \ \overline{X}_2 \ e^{-i2\mathbb{A}}) + \mod{X_1}^2 +
\mod{X_2}^2 +1].
$$
\end{prop}

\begin{prop} The determinants of the submatrices $G(i,j,k)$ of $G$ are given by

$$\det G(1,2,3)= -2{\rm Re}(e^{i \mathbb{A}}), \ \ \ \ \ \ \ \ \ \ \ \
\ \ \ \ \ \ \ \ \det G(1,2,4)=-2{\rm Re}\Bigl(
\frac{\overline{X}_1}{\mod{X_2}^2} \ e^{i \mathbb{A}}\Bigr), $$

$$ \  \  \ \  \det G(1,3,4)=-2{\rm Re}\Bigl(
\frac{\overline{X}_2}{\mod{X_2}^2} \ e^{-i \mathbb{A}}\Bigr),\  \ \
\ \ \ \ \ \  \det G(2,3,4)=-2{\rm Re}\Bigl( \frac{X_1 \
\overline{X}_2}{\mod{X_2}^2} \ e^{-i \mathbb{A}}\Bigr).$$
\end{prop}

\section{The configuration space of ordered quadruples in the
boundary of complex hyperbolic space and its moduli space}

\subsection{The moduli space of ordered quadruples in the
boundary of the complex hyperbolic plane}

Let $\mathcal M $ be the configuration space of ordered quadruples
of distinct points in the boundary of the complex hyperbolic plane,
that is, the quotient of the set of ordered quadruples of distinct
points of $\partial\ch{2}$ with respect to the diagonal action of
${\rm PU}(2,1)$ equipped with the quotient topology. In this
section, we construct a moduli space for $\mathcal M$.

\vspace{4mm}
The following proposition is the crucial result in our
construction of the moduli space for $\mathcal M$.

\begin{prop} Let $G=(g_{ij})$ be an Hermitian $4\times 4$-matrix such that
$g_{ii}=0$, $g_{12}=g_{23}=g_{34}=1$, $\mod {g_{13}}=1$, $g_{14}
\neq 0,$ $g_{24} \neq 0.$ Then $G$ is the normalized Gram matrix for
some ordered quadruple $p=(p_1,p_2,p_3,p_4)$ of distinct points of
$\partial\ch{2}$ if and only if {\rm Re}$(g_{13})\leq 0$, {\rm
Re}$(g_{24} \bar{g}_{14})\leq 0$, and $\det G=0.$
\end{prop}
\begin{proof} Let us assume that $G$ is the normalized Gram matrix associated
to an ordered quadruple $p=(p_1,p_2,p_3,p_4)$ of distinct points of
$\partial\ch{2}.$ Then it follows from Proposition 2.4 that
$$ \  \ \ \det G(1,2,3)=2 {\rm Re}{(\overline{g}_{13})},\ \ \ \ \
\ \ \det G(1,2,4)=2{\rm Re} (g_{24} \ {\overline{g}_{14}}).$$ By
Sylvester's Criterion, these determinants are negative or vanish.
Since $p$ is in the boundary of the complex hyperbolic space of
dimension $2$, any vectors $P_1,P_2,P_3,P_4$ representing
$p_1,p_2,p_3,p_4$ are linearly dependent. This implies that $\det
G=0.$

Now let $G=(g_{ij})$ be an Hermitian $4\times 4$-matrix such that
$g_{ii}=0$, $g_{12}=g_{23}=g_{34}=1$, $\mod {g_{13}}=1,$  ${\rm
Re}(g_{13})\leq 0$, ${\rm Re}{(g}_{24} \bar{{g}}_{14})\leq 0$, and
$\det G =0$. We are going to show that there exist four null
(isotropic) vectors $P_1,P_2,P_3,P_4,$ $P_i \in \mathbb {C}^{2,1},$
whose Gram matrix is equal to $G$.

\vspace{2mm}

We will look for these vectors in the following form:

$$P_1 \ = \ \begin{bmatrix} 0 \\ 0 \\  1 \end{bmatrix},
\ \ \ \  P_2 \ = \ \begin{bmatrix} 1 \\  0 \\ 0 \end{bmatrix}, \ \ \
P_3 \ = \ \begin{bmatrix} z_1 \\ z_2 \\ 1 \end{bmatrix}, \ \ \  P_4
\ = \   \begin{bmatrix} w_1 \\ w_2 \\ w_3 \end{bmatrix}, $$ where
$z_i$, $w_i$ are complex numbers, and $\mod{z_1}=1.$

\vspace{2mm}

Then we have
$$\herm{P_1,P_2}=1, \ \herm{P_1,P_3}=\bar{z}_{1}, \
\herm{P_1,P_4}=\bar{w}_{1}, \ \herm{P_2,P_3}=1, \
\herm{P_2,P_4}=\bar{w}_3, \ \herm{P_3,P_4}=z_1 \bar{w}_3 + z_2
\bar{w}_2 +\bar{w}_1.
$$
Since we need $P_3$ and $P_4$ to be null vectors, we have the
following equations:
$$ z_1 + \mod{z_2}^2 + \bar{z}_1 =0, \ \ \ \ \ \ \ w_1 \bar{w}_3 + \mod{w_2}^2 +
\bar{w}_1 w_3 =0.
$$
From the definition of the Gram matrix, we have
$$
g_{12}=1,\ \  g_{13}=\bar{z}_1 ,\ \  g_{14}=\bar{w}_1 , \ \
g_{23}=1, \ \  g_{24}= \bar{w}_3 ,
$$
and
$$
g_{34}= z_1 \bar{w}_3 + z_2 \bar{w}_2 + \bar{w}_1 = 1.
$$
This implies that we have already found $z_1 , w_1, w_3 $ in terms
of $g_{ij}$. Therefore, we need to find $z_2$ and $w_2$.

\vspace{2mm}
We consider the following system of equations
$$ (1)\ \ z_1 + \mod{z_2}^2 + \bar{z}_1 =0, \ \ \ \ \
(2) \  \  w_1 \bar{w}_3 + \mod{w_2}^2 + \bar{w}_1 w_3 =0, \ \ \ \ \
(3) \ \  z_1 \bar{w}_3 + z_2 \bar{w}_2 + \bar{w}_1 = 1, $$ and show
that it has a solution under the conditions of the proposition. We
write the first two equations in the following form:
$$|z_2|^2 = -2 {\rm Re} (z_1) = -2{\rm Re}(g_{13}), \ \ \ \ \
|w_2|^2 = -2{\rm Re}(w_1 \bar{{w}}_3) = -2{\rm Re}(g_{24} \bar{{\mit
g}}_{14}).$$ We immediately get that there exist solutions for
$\mod{z_2}$ and $\mod{w_2}$ under our conditions.

\vspace{2mm}

The third equation may be written as
$$   z_2 \bar{w}_2 = 1 \ - \  z_1 \bar{w}_3 \  -  \ \bar{w}_1.$$
Let us assume that the equation $\mod{z_2 \bar{w}_2} = \mod{1 \ - \
z_1 \bar{w}_3 \  -  \ \bar{w}_1}$ holds for some $z_2$ and $w_2$
satisfying equations $(1)$ and $(2)$, and let $(z_2,w_2)$ be any
solution to this equation. First, we suppose that $z_2 \neq 0, w_2
\neq 0$. In this case, we put $z'_2 \ = \ e^{i\alpha} \ z_2.$ Then
$z'_2 \ \bar{w}_2 =\ e^{i\alpha} \ ( z_2 \bar{w}_2).$ Since two
complex numbers have the same norm if and only if there exists a
rotation which sends one number to another, we have that there
exists $\alpha$ such that $ z'_2 \bar{w}_2 = 1 \ - \ z_1 \bar{w}_3 \
-  \ \bar{w}_1.$ Therefore, the system above has solutions if and
only if
 $$A = \
\mod{1 \ - \ z_1 \bar{w}_3 \  -  \ \bar{w}_1}^2 \  - \ \mod{z_2
\bar{w}_2}^2 \ =0.$$ Substituting $\mod{z_2}$ and $\mod{w_2}$ from
equations (1) and (2) into  $A,$ and then rearranging, we have that
$$A = \ \mod{1 \ - \ z_1
\bar{w}_3 \  -  \ \bar{w}_1}^2 \ - \ \mod{z_2 \bar{w}_2}^2 =
$$
$$
(1 \ -  z_1 \bar{w}_3 \  -   \bar{w}_1) \ \ (1 \ -  \bar {z}_1  \
w_3 \ - \ w_1) \ - (z_1 + \bar{z}_1) \  (w_1 \ \bar{w}_3 \ \ + w_3 \
\bar{w}_1) =
$$
$$1 + \mod{z_1}^{2}  \mod{w_3}^{2} \ + \ \mod{w_1}^{2} \
-2 {\rm Re}(w_1) \  -2 {\rm Re} ( z_1 \ \bar{ w}_3) -2 {\rm Re} (
z_1  w_3 \bar{w}_1).
$$
Since  $g_{13}=\bar{z}_1 ,\ \ g_{14}=\bar{w}_1 ,  \ \ g_{24}=
\bar{w}_3,$ and $\mod {g_{13}}=1,$ we get the following expression
for $A$ in terms of $g_{ij}$:

$$A= -2{\rm Re}(g_{14}) \ -2{\rm Re}(g_{13} \bar{g}_{24}) \
- 2{\rm Re}(g_{13} \bar{g}_{14} g_{24}) + \mod{g_{14}}^2 +
\mod{g_{24}}^2 +1.
$$
By applying Proposition 2.3, we see that $A=\det G$. Since, by our
hypothesis $\det G =0$, we get the result we need in the case $z_2$
and $w_2$ are not equal to zero. It is easy to see that there exists
a solution to the system $(1)-(3)$  when one of the numbers $z_2$ or
$w_2$ is equal to zero, since in this case the third equation is
satisfied automatically provided that $A=0$. Finally, one sees that
if $g_{14} \neq 0,$ $g_{24} \neq 0,$ then the points $p_i$ defined
by the vectors $P_i$ are distinct. This proves the proposition.
\end{proof}
\begin{co}Let $G=(g_{ij})$ be an Hermitian $4\times 4$-matrix such that $g_{ii}=0$,
$g_{12}=g_{23}=g_{34}=1$, $\mod {g_{13}}=1,$ $g_{14} \neq 0,$
$g_{24} \neq 0,$ and $\det G=0.$  Then the inequalities {\rm
Re}$(g_{13})\leq 0$ and {\rm Re}$(g_{24} \bar{g}_{14})\leq 0$ imply
the inequalities {\rm Re}$(g_{13} \bar{g}_{14})\leq 0$ and {\rm
Re}$(\bar{g}_{24})\leq 0.$
\end{co}
\begin{proof} It follows from Proposition 3.1 that under the conditions
above $G$ is the normalized Gram matrix for some ordered quadruple
of distinct points of $\partial\ch{2}.$ Then the result follows from
Proposition 2.4.
\end{proof}

\vspace{2mm}
\begin{co} Let $G=(g_{ij})$ be an Hermitian $4\times 4$-matrix
satisfying all the conditions of Proposition 3.1. Then any two
quadruples of points $p=(p_1,p_2,p_3,p_4)$ and
$p'=(p'_1,p'_2,p'_3,p'_4)$ of $\partial\ch{2}$ defined by the
quadruples of null vectors $P=(P_1,P_2,P_3,P_4)$ and
$P'=(P'_1,P'_2,P'_3,P'_4)$ that correspond to solutions  with $(z_2,
w_2)$ and $(z'_2, w'_2)$ respectively are congruent in ${\rm
PU}(2,1).$
\end{co}
\begin{proof} It follows from the proof of Proposition 3.1 that the
quadruples $ p=(p_1,p_2,p_3,p_4)$ corresponding to the vectors
$P_1,P_2,P_3,P_4$ defined there  have the same normalized Gram
matrix. By applying Proposition 2.2, we get the result we want.
\end{proof}

\vspace{2mm}

Let $p=(p_1,p_2,p_3,p_4)$ be an ordered quadruple of distinct points
of $\partial\ch{2}$, and let $m(p)\in \mathcal M $ be the point
represented by $p=(p_1,p_2,p_3,p_4)$.  We define the map
$$\tau : \mathcal M \longrightarrow  \mathbb C^{2} \times \mathbb R$$
by the following formula:

\begin{equation*}
\dps \tau :m(p)\mapsto  \bigl (X_1=X(p_1,p_2,p_3,p_4),\ X_2=
X(p_1,p_3,p_2,p_4),\ \mathbb A= \mathbb{A} (p_1,p_2,p_3)\bigr),
\end{equation*}
where  $\mathbb A= \mathbb{A} (p_1,p_2,p_3)$ is the Cartan invariant
of the triple $(p_1,p_2,p_3)$.

\begin{prop} Let $X_1, X_2, \mathbb A$ be the numbers
defined by $\tau$. Then they satisfy the following relation:

\begin{equation*}
\dps -2{\rm Re}(X_1+X_2) -2{\rm Re} (X_1\ \overline{X}_2\ e^{-
i2\mathbb A})+ \mod {X_1}^2 +\mod {X_2}^2 +1=0.
\end{equation*}
\end{prop}

\begin{proof}
Since $p \in \partial\ch{2}$,  the determinant of the Gram matrix
$G=G(p)$ associated to $P_i$, where $P_i \in V^{2,1}$ is a lift of
$p_i$, is equal to zero because $ P_1, P_2, P_3, P_4$ are linearly
dependent. Then the result follows from Proposition 2.7.
\end{proof}

\vspace{2mm}
 We are now ready to prove our main result.
\begin{theorem}
The configuration space $\mathcal M $ is homeomorphic to the set of
points  $X=(X_1, X_2, A) \ \in \ \mathbb C_{*}^2 \times \mathbb R$
defined by
\begin{equation*} \dps -2{\rm Re}(X_1+X_2) -2{\rm Re}
(X_1\ \overline{X}_2\ e^{-\ i2 A})+ \mod {X_1}^2 +\mod {X_2}^2 +1=0,
\end{equation*}
subject to the following conditions
$$\dps -\pi /2 \leq A \leq \pi/2, \ \ \ {\rm Re} (X_1 \ e^{-i A})\geq
0.$$
\end{theorem}

\vspace{2mm}

\noindent {\bf Remark} We denote this set by $\mathbb M$ and call
$\mathbb M$ the {\sl moduli} space for $\mathcal M$.

\vspace{2mm}

\begin{proof} It follows from Proposition 3.2 and the formulae in
Proposition 2.8 that the map $ \tau $ above defines a map $ \tau :
\mathcal M \longrightarrow \mathbb M.$ First we show that the map $
\tau : \mathcal M \longrightarrow \mathbb M$ is surjective. Given
$X=(X_1, X_2, A)  \in \mathbb M$, we construct an Hermitian $4\times
4$-matrix $G=(g_{ij})$ as follows. Using the formulae in Proposition
2.6, we define $g_{13}$, $g_{14}$, $g_{24}$ in terms of $X_1, X_2,
A$, that is, we put
$$ \ \ \ g_{13}=-e^{-i A}, \ \ \  \ g_{14}=\frac{1}{\overline{X}_2}, \ \ \
g_{24}=-\frac{\overline{X}_1}{\overline{X}_2} \  \ e^{i A}.$$ Also
we put $g_{ii}=0$, $g_{12}=g_{23}=g_{34}=1.$ This defines $G$
completely.  Comparing the formulae in Propositions 2.3 and 2.4 with
those in Propositions 2.7 and 2.8, we see that $G$ satisfies all the
conditions in Proposition 3.1. By applying Proposition 3.1, we get
that $G$ is the normalized Gram matrix for some ordered quadruple
$p=(p_1,p_2,p_3,p_4)$ of distinct points of $\partial\ch{2}$. It is
readily seen from the formulae in Proposition 2.6 that $\tau (m(p))=
(X_1, X_2, A).$ This proves that $\tau$ is surjective. On the other
hand, it follows from Corollary 2.1 and Corollary 2.2 that $\tau$ is
injective. It is clear that $ \tau : \mathcal M \longrightarrow
\mathbb M$ is a homeomorphism provided that $\mathbb M$ is equipped
with the topology induced from $\mathbb C^{2} \times \mathbb R$.
This completes the proof of the theorem.
\end{proof}

\vspace{2mm}

\noindent {\bf Remark} \ In Section 3.2.5, we will show that the
equation
\begin{equation*}
\dps -2{\rm Re}(X_1+X_2) -2{\rm Re} (X_1\ \overline{X}_2\ e^{-\ i2
A})+ \mod {X_1}^2 +\mod {X_2}^2 +1=0,
\end{equation*}
and the strong inequalities $\dps -\pi /2 < A < \pi /2$ imply the
inequality ${\rm Re} (X_1 \ e^{-i A})\geq 0.$

\subsection{Some interesting subsets of $\mathbb M$ and the topology of the moduli space}

\subsubsection{Singular set of $\mathbb M$ and $\mathbb C$-plane
configurations}

The topology of the moduli space $\mathbb M$ is quite interesting.
To describe it, we define the following sets.

\vspace{3mm}

Let $\mathbb S$ be the set of points in $\mathbb C^{2} \times
\mathbb R,$ where the equation
\begin{equation*}
\dps -2{\rm Re}(X_1+X_2) -2{\rm Re} (X_1\ \overline{X}_2 \ e^{- \ i2
A})+ \mod {X_1}^2 +\mod {X_2}^2 +1=0
\end{equation*}
is satisfied. We call this set $\mathbb S$ the {\sl basic variety}.

\vspace{2mm}

Let also
\begin{itemize}
\item $\mathbb{S}_{123}=\{(X_1, X_2, A) \in \mathbb C^{2} \times \mathbb R: {\rm Re}(e^{i A}) =0\},$

\item $\mathbb{S}_{124}=\{(X_1, X_2, A) \in \mathbb C^{2} \times \mathbb R: {\rm Re}(
\overline{X}_1 \ e^{i A}) =0\},$

\item $\mathbb{S}_{134}=\{(X_1, X_2, A)\in \mathbb C^{2} \times \mathbb R:  {\rm Re}(
\overline{X}_2 \ e^{-iA}) = 0\},$

\item $\mathbb{S}_{234}=\{(X_1, X_2, A)\in \mathbb C^{2} \times \mathbb R: {\rm Re}(X_1 \
\overline{X}_2 \ e^{-iA})=0\}.$
\end{itemize}

\vspace{2mm}

It is easy to see that all the varieties $\mathbb{S}_{ijk}$ are not
singular. In fact, $\mathbb{S}_{ijk}$ is diffeomorphic to the
disjoint union of $4$-planes in $\mathbb C^{2} \times \mathbb R$.
For instance,
$$\mathbb{S}_{123}=\{(X_1, X_2, A) \in \mathbb C^{2} \times \mathbb R :
A=\pm \pi/2 +2k \pi, \ k \in \mathbb Z \}. $$

\vspace{2mm}

We call $\mathbb{S}_{ijk}$ the {\sl Cartan varieties}.

\vspace{2mm}

Let $p=(p_1,p_2,p_3,p_4)$ be an ordered quadruple of distinct points
of $\partial\ch{2}.$ We call $p$  a {\sl tetrahedron} with {\sl
vertices } $(p_1,p_2,p_3,p_4)$, and also we call the triples
$(p_i,p_j,p_k)$ the {\sl faces} of $p$.

\vspace{2mm}

We say that a tetrahedron $p$ is {\sl almost} $\mathbb C$-{\sl
plane} if there exists a face of $p$ which lies on a chain (we call
such a face $\mathbb C$-{\sl plane}). Also, we define a tetrahedron
$p$ to be $\mathbb C$-{\sl plane} if all of its vertices are in a
chain.

\vspace{2mm} Given tetrahedron $p$, let $m(p) \in \mathcal M$ be the
point represented by $p$. Then we have the following description of
almost $\mathbb C$-plane and $\mathbb C$-plane tetrahedra.
\begin{prop} A tetrahedron $p$ is almost $\mathbb C$-plane if and only
if $\tau (m(p))$ belongs to some Cartan's variety. Moreover, a
tetrahedron $p$ is $\mathbb C$-plane if and only if $\tau (m(p))$
belongs to the intersection of at least two Cartan's varieties.
\end{prop}
\begin{proof} The first assertion is obvious because of Proposition 2.8.
The second one follows from the fact that if two chains intersect in
at least two points then they coincide. In particular, a tetrahedron
is $\mathbb C$-plane if and only if it has at least two $\mathbb
C$-plane faces.
\end{proof}

\vspace{2mm}

Let $\mathcal C$ be the set of the points in $\mathbb M$ which
represent $\mathbb C$-plane tetrahedra. In what follows, we will
describe the singular set of $\mathbb S$ when $-\pi/2 \leq  A
 \leq \pi/2$. Also, we will find a
relation between $\mathcal C$ and the singular set of $\mathbb M.$

\vspace{2mm}

\begin{prop}
The basic variety $\mathbb S$ has no singular points for $-\pi/2 < A
 < \pi/2$ and $X_1 \neq 0, \ X_2 \neq 0$.
\end{prop}

\vspace{2mm}
\begin{proof}
We write $X_1 = a +bi$, $X_2 = c +di$, $e^{-i2 A} = \cos{2A} - i
\sin{2A}.$ Then easy calculations show that $\mathbb S$ in terms of
$a,b,c,d,A$ is given by the following equation
$$F(a,b,c,d,A) = -2(a+c) -2[ (ac +bd) \cos{2A} + \ (bc-ad) \sin{2A} ] +
a^2 + b^2 +c^2 + d^2 +1 = 0.$$

Computing the partial derivatives of $F$, we get the following
system for finding singular points of $\mathbb S.$

\vspace{2mm}

$(1) \ \frac{\partial{F}}{\partial{a}}=0 \  \Longleftrightarrow  \
-1 -c \cos{2A} + d \sin{2A} + a =0,$

\vspace{2mm}

$(2) \ \frac{\partial{F}}{\partial{c}}=0 \ \Longleftrightarrow \ -1
- a \cos{2A} - b \sin{2A} + c =0,$

\vspace{2mm}

 $(3) \ \frac{\partial{F}}{\partial{b}}=0 \
\Longleftrightarrow \  - d \cos{2A} - c \sin{2A} + b =0,$

\vspace{2mm}

$(4) \ \frac{\partial{F}}{\partial{d}}=0 \ \Longleftrightarrow \  -
b \cos{2A} + a \sin{2A} + d =0,$

\vspace{2mm}

$(5) \ \frac{\partial{F}}{\partial{A}}=0 \ \Longleftrightarrow \
  (ac +bd) \sin{2A} -(bc-ad) \cos{2A} =0.$
\vspace{2mm}

Since $X_1 \neq 0$ and $X_2 \neq 0$, we have that $a^2 + b^2 \neq 0$
and $c^2 + d^2 \neq 0.$ We consider the following cases: $a=0, b\neq
0$; $a \neq 0, b=0$; $a\neq 0, b\neq 0.$ In the two first cases, the
equations become very simple, and a straightforward verification,
which is left to the reader, shows that the system above has no
solutions for $-\pi/2 < A <\pi/2.$ So, in these  cases $\mathbb S$
has no singular points.

Let us assume now that $a\neq 0, b\neq 0.$ By multiplying the first
equation by $b$ and the third equation by $(-a),$ and then summing
the resulting equations, we get that $(ac +bd) \sin{2A} -(bc-ad)
\cos{2A} -b =0.$ Now, if equation $(5)$ does not hold, then the
gradient of $F$ is not the null vector. If equation $(5)$ holds,
then the above implies $b=0$, that contradicts our assumption. This
completes the proof of the proposition.
\end{proof}

\vspace{2mm}

Next we show that $\mathbb S$ has singular points for $ A = \pm
\pi/2 $ and describe the singular set of $\mathbb M$.

\begin{prop} The basic variety $\mathbb S$ has singular points for $ A = \pm
\pi/2.$
\end{prop}
\begin{proof}
Substituting $ A = \pm \pi/2$ into equations $(1)-(5)$, it is
readily seen that the system has solutions if and only if $a+c=1$
and $b=d=0.$ On the other hand, using the equation from Proposition
3.4, we have that $\mathbb S$ is given for  $ A = \pm \pi/2$ by the
following equation
$$F(a,b,c,d) = -2(a+c) +2(ac +bd) + a^2 + b^2 +c^2 + d^2 +1 = 0.$$
Rearranging this gives the equation
$$ (a+c -1)^2 +(b+d)^2 =0.$$
Thus, for $ A = \pm \pi/2$, \  $\mathbb S$ is defined by the
following system of equations
$$ a+c = 1, \ \ \ \  b+d =0.$$
All this implies that in all the points of $\mathbb S$, where $ A =
\pm \pi/2$, $a+c=1$ and $b=d=0$, the gradient of $F$ is the null
vector, and, therefore, all these points are singular points of
$\mathbb S$.
\end{proof}

\vspace{2mm}

As a corollary of the proof of Proposition 3.5, we obtain the
following
\begin{prop}
The intersection of the Cartan variety $\mathbb{S}_{123}$ with the
basic variety $\mathbb S$ is not transversal. This intersection is
the union of $2$-dimensional real analytic varieties. In the
coordinates above, it is given by the equations
$$a+c = 1, \ b+d =0,  \ A = \pm \pi/2 +2k \pi.$$
\end{prop}

\vspace{2mm}

Using calculations similar to those in Proposition 3.5,
we get the following

\begin{prop}
The intersection of the Cartan variety $\mathbb{S}_{124}$ with the
basic variety $\mathbb S$ is not transversal. This intersection is
the union of $2$-dimensional real analytic varieties. In the
coordinates above, it is given by the equations
$$ a + b\tan A = 0, \   a+c = 1, \ b-d =0,$$
provided that $ A \neq \pm \pi/2 +2k \pi$, and by the equations
$$a+c=1, \ \ b=0, \ \ d=0,$$
provided that $ A = \pm \pi/2 +2k \pi.$
\end{prop}

Also, we have
\begin{prop}
The intersection of the Cartan variety $\mathbb{S}_{134}$ with the
basic variety $\mathbb S$ is not transversal. This intersection is
the union of $2$-dimensional real analytic varieties. In the
coordinates above, it is given by the equations
$$ c - d\tan A = 0, \   a+c = 1, \ b-d =0,$$
provided that $ A \neq \pm \pi/2 +2k \pi$, and by the equations
$$a+c=1, \ \ b=0, \ \ d=0,$$
provided that $ A = \pm \pi/2 + 2k \pi.$
\end{prop}

\begin{prop}
The intersection of the Cartan variety $\mathbb{S}_{234}$ with the
basic variety $\mathbb S$ is not transversal. This intersection is
the union of  $2$-dimensional real analytic varieties. In the
coordinates above, it is given by the equations
$$ (ac+bd) + (bc-ad) \tan A = 0, \   a+c = 1, \ b+d =0,$$
provided that $ A \neq \pm \pi/2 + 2k\pi$, and by the equations
$$a+c=1, \ \ b=0, \ \ d=0,$$
provided that $ A = \pm \pi/2 + 2k\pi.$
\end{prop}

Now let
$$\mathbb{S}_1 = \mathbb S \cap \mathbb{S}_{123}, \ \ \mathbb{S}_2 = \mathbb S \cap
\mathbb{S}_{124}, \ \  \mathbb{S}_3 = \mathbb S \cap
\mathbb{S}_{134}, \ \ \mathbb{S}_4 = \mathbb S \cap
\mathbb{S}_{234}.$$
Thus,  we get that $\mathbb{S}_i$ is a
$2$-dimensional real analytic variety, and, moreover, we have the
following
\begin{co}
The intersection of all $\mathbb{S}_i$ is exactly the singular set
of the basic variety $\mathbb S.$
\end{co}
\begin{co}
The set $\mathcal C$ is equal to the singular set of the moduli
space $\mathbb M$.
\end{co}

\subsubsection{$\mathbb C$- plane configurations}
In this section, we describe the set of $\mathbb C$-plane tetrahedra
in terms of our coordinates $X_1, X_2, A$.

\begin{theorem}
Let $p$ be a tetrahedron, and let $ \tau (m(p)) = (X_1, X_2,A)$.
Then $p$ is $\mathbb C$-plane if and only if $A = \pm \pi/2 $ and
$X_1, X_2$ are real numbers satisfying the relation $X_1 + X_2=1$.
\end{theorem}

\begin{proof}
If $p$ is $\mathbb C$-plane then it follows immediately from the
proof of Proposition 3.5 that $X_1, X_2,A$ satisfy all the
conditions of the theorem. Assume now that $X_1, X_2,A$ satisfy the
conditions of the theorem. Then it is easy to see that they satisfy
all the conditions of Theorem 3.1. Therefore, there exists a
tetrahedron $p= (p_1,p_2,p_3,p_4)$ such that $ \tau (m(p)) = (X_1,
X_2,A)$. Since $A= \pm \pi/2$, the points $p_1,p_2,p_3$ are in a
chain. Moreover, it follows from the proof of Proposition 3.5 that
the points $p_1,p_2,p_4$ are also in a chain. This implies that the
points $p_1,p_2,p_3,p_4$ must be in the same chain.
\end{proof}

\subsubsection{$\mathbb R$- plane configurations}

Let $p=(p_1,p_2,p_3,p_4)$ be a tetrahedron of distinct points of
$\partial\ch{2}$. We say that $p$ is $\mathbb R$-{\sl plane} if all
of its vertices are in an $\mathbb R$-circle.

\begin{theorem}
Let $p$ be a tetrahedron, and let $ \tau (m(p)) = (X_1, X_2,A)$.
Then $p$ is $\mathbb R$-plane if and only if $A = 0$ and $X_1, X_2$
are positive real numbers satisfying the relation
$$-2(X_1 + X_2) - 2 \ X_1  X_2 + X_1^2 + X_2^2 +1 =0.$$
\end{theorem}
\begin{proof} It follows immediately from Theorem 3.1  that if $p$ is $\mathbb
R$-plane, then $A=0$ and $X_1, X_2$ are real positive numbers
satisfying the relation in the theorem. Let us assume now that $A=0$
and $X_1, X_2$ are real numbers satisfying the conditions of the
theorem. Then it is readily seen that they satisfy all the
conditions of Theorem 3.1. Therefore, there exists a tetrahedron $p=
(p_1,p_2,p_3,p_4)$ such that $ \tau (m(p)) = (X_1, X_2,0)$. Since
$X_1, X_2$ are real and $A=0$, by applying the formulae in
Proposition 2.6, we get that the normalized Gram matrix associated
to $p$ has real coefficients. Then it follows from Corollary 2.2
that there exists an anti-holomorphic involution which fixes every
point $p_i$, $i=1,2,3,4.$ So, these points must be in the same
$\mathbb R$-circle.
\end{proof}

\subsubsection{Real slice of the basic variety $\mathbb S$}

We call the subset $\mathcal R$ of the basic variety $\mathbb S$ the
{\sl real slice} of  $\mathbb S$ if and only if $X_1$ and $X_2$ are
real numbers.
\begin{theorem} Any point in the real slice $\mathcal R$ satisfies
the inequality ${\rm Re} (X_1 \ e^{-i A})\geq 0$ for $\dps -\pi /2
\leq A \leq \pi/2.$
\end{theorem}

\begin{proof}
We write $X_1 = a +bi$, $X_2 = c +di$, $e^{-i2 A} = \cos{2A} - i
\sin{2A}.$  It is readily seen that $\mathbb S$ in terms of
$a,b,c,d,A$ is given by the following equation
$$F(a,b,c,d,A)=F(a,c,A) = -2(a+c) -2(ac) \cos{2A} +
a^2 +c^2 +1 = 0,$$ provided that $b=d=0$. We consider this equation
as an equation in $a$ and $c,$ where $A$ is a parameter. Thus, we
have a family of conics. Easy calculations show that the
discriminant $D$ of any conic in this family is equal to
$D=-4\sin^2{2A}.$ Therefore, these conics are of elliptic or
parabolic type. Moreover, easy considerations show that any conic in
this family is either ellipse or parabola for any $\dps -\pi /2 < A
< \pi/2$ (which are tangent to the axes at the points $(a,c)=(1,0)$
and $(a,c)=(0,1)$), or the double line $(a+c-1)^2=0$ for $A=\pm
\pi/2$ (compare this with the equations in Propositions 3.4  - 3.5).
This implies that if $(X_1, X_2, A)$ lies in the real slice
$\mathcal R$ then $a>0$ or $A=\pm \pi/2$. So, in this case, ${\rm
Re} (X_1 \ e^{-i A}) = a\cos A \geq 0.$ This proves the theorem.
\end{proof}
\begin{co} Any point in the real slice $\mathcal R$ represents a
point in the configurations space $\mathcal M$ provided that $\dps
-\pi /2 \leq A \leq \pi/2.$ The  $\mathbb C$-plane and $\mathbb
R$-plane tetrahedra define a subset lying in the real slice.
\end{co}

\subsubsection{Topological picture of the moduli space $\mathbb M$}

In this section, we describe the topology of the moduli space
$\mathbb M.$

\vspace{2mm}

It follows from Theorem 3.1 that not all points of $\mathbb{S}_1$
with $A = \pm \pi/2$ belong to the moduli space $\mathbb M$: the
singular set $\mathcal C$ divides this set into two parts, and only
the part defined by the restrictions in Theorem 3.1 belongs to
$\mathbb M$. We call such a part {\sl positive}. Using notations as
in Theorem 3.1, it is easy to see that the positive part is defined
by the inequality $b\geq 0$, when $ A = \pi/2$, and by the
inequality $b \leq 0$, when $ A = - \pi/2.$

\vspace{2mm}

Another important observation which may help to understand the
topology of $\mathbb M$ is the following.

\begin{prop} Any point $(X_1, X_2, A)$ in the basic variety $\mathbb S$
satisfies the inequality ${\rm Re} (X_1 \ e^{-i A})\geq 0$ for $\dps
-\pi /2 < A < \pi/2.$
\end{prop}

\begin{proof}
Let us write the equation for the basic variety $\mathbb S$ as in
the proof of Proposition 3.4, that is,
$$F(a,b,c,d,A) = -2(a+c) -2[ (ac +bd) \cos{2A} + \ (bc-ad) \sin{2A} ] +
a^2 + b^2 +c^2 + d^2 +1 = 0.$$ We consider this equation as an
equation in $a$ and $b,$ where $c,d,A$ are considered as parameters.
Thus, we have a family of conics. It is easy to see that the
discriminant $D$ of any conic in this family is equal to $D=-1.$
Therefore, all these conics are of elliptic type. We are going to
determine the relative position of any conic in this family and the
line $L$ given by the equation ${\rm Re} (X_1 \ e^{-i A}) = a\cos A
+b \sin A =0$ for $\dps -\pi /2 < A < \pi/2.$ Let $K$ be a conic in
this family. We assume first that $K$ is not degenerate, that is,
$K$ is an ellipse. Using Proposition 3.8, we see that $K$ and $L$
are tangent, and, therefore, $K$ intersects $L$ in only one point
which is a unique point of tangency of $K$ and $L$. Hence, $K$ lies
entirely in the closure of a component of the complement of $L$.
Then easy arguments show that $K$ lies in the component where $a
\cos A +b \sin A \geq 0.$ On the other hand, since $\mathbb S$ has
no singular points for $\dps -\pi /2 < A < \pi/2,$ see Proposition
3.8, it readily seen that $K$ is not degenerate for $\dps -\pi /2 <
A < \pi/2.$  All this implies that for $\dps -\pi /2 < A < \pi/2$ \
the basic variety $\mathbb S$ lies in the component of the
complement of the Cartan variety $\mathbb S_{134}$ where ${\rm Re}
(X_1 \ e^{-i A}) \geq 0.$ This proves the proposition.
\end{proof}

\vspace{2mm}

Resuming all the above, we get the following  topological picture
for $\mathbb M$: the moduli space $\mathbb M$ looks like a real
analytic $4$-dimensional variety in $\mathbb C_{*}^{2} \times
\mathbb R$ truncated by a $4$-dimensional analytic variety
intersecting this variety non-transversally along a $2$-dimensional
analytic subvariety. The truncated part contains the singular set of
$\mathbb M$ which is a $1$-dimensional real analytic subvariety.

\vspace{2mm}

This implies that the moduli space $\mathbb M$ cannot be described
as a real analytic (algebraic) variety.

\vspace{2mm}

\subsection{The configuration space of ordered quadruples in the
boundary of the complex hyperbolic plane and the
Falbel-Parker-Platis cross-ratio varieties}

In this section, we find a relation between our moduli space
$\mathbb M$ and the cross-ratio varieties constructed by Falbel and
Parker-Platis. We show that their varieties cannot serve as moduli
spaces for the configuration space $\mathcal M$. We consider only
the Parker-Platis variety which is more easy to describe. We remark
that Falbel' variety is homeomorphic to the Parker-Platis variety,
see  \cite{PP2}.

\vspace{2mm}

In \cite{PP1}, Parker and Platis  define the map $\pi' : \mathcal M
\longrightarrow \mathbb C^{3}$ by the following formula:
\begin{equation*}
\pi' :m(p)\rightarrow \bigl (X_1=X(p_1,p_2,p_3,p_4), \ X_2=
X(p_1,p_3,p_2,p_4),\  X_3=X(p_2,p_3,p_1,p_4)\bigr).
\end{equation*}

\vspace{2mm}

They proved that $X_1, X_2, X_3$ satisfy the following relations:

\begin{itemize}
\item $| X_2| =| X_1| | X_3|,$
\item $2| X_1| ^2{\rm Re}(X_3)=| X_1| ^2+| X_2| ^2+1-2 \ {\rm Re}(X_1+X_2).$
\end{itemize}

\vspace{2mm}

The Parker-Platis cross-ratio variety $\mathcal X'$ is the subset of
$\mathbb{C}^3,$ where these relations are satisfied. The main result
of \cite{PP1} is that the map $\pi' : \mathcal M \longrightarrow
\mathcal X' $ defined above is a bijection, see Proposition 5.5 and
Proposition 5.10 in \cite{PP1}. Next we show that this map is not
injective.

\vspace{2mm}

 We define the map $\theta :
\mathbb{M} \rightarrow \mathbb{C}^3$  by the formula
$$\theta : (X_1, X_2,A) \mapsto (X_1,\ X_2,\ (X_{2}/X_{1}) \
e^{2iA}).$$  It follows from Proposition 2.6 that $ X_3 =
(X_{2}/X_{1}) \ e^{2iA}.$ Also, easy calculations show that for
every $(X_1, X_2,A) \in \mathbb{M}$  the point $\theta (X_1, X_2,A)$
belongs to $\mathcal X'.$ In fact, the first equation in the
definition of the Parker-Platis cross-ratio variety is exactly the
property of the Kor\'{a}nyi-Reimann complex cross-ratios which
relates $X_1, X_2, X_3$, see p.225 in  Goldman \cite{Go}. The second
relation is equivalent to the equation $\det G =0$, where $G$ is a
Gram matrix associated to $p$, and the Gram determinant is expressed
in terms of $X_1, X_2, X_3$, see Proposition 2.6. Thus, $\theta$
defines a map $\theta : \mathbb{M} \rightarrow \mathcal X'.$

It follows from the definition that $\theta$ is injective for
$-\pi/2 < A < \pi/2.$ On the other hand, it is easy to see that
$\theta$ is not injective being restricted to the subset of $\mathbb
{M}$, where $A=\pm \pi/2$. Using the results of Section 3.2.1, we
see that $\theta$ is not injective exactly on the singular set
$\mathcal C$, that is,
$$\theta (X_1, X_2, \pi/2)=\theta (X_1, X_2, -\pi/2)$$
if and only if $X_1, X_2$ satisfy the conditions of Theorem 3.1.
Thus, $\theta$ glues the points of $\mathbb M$ corresponding to the
$\mathbb C$-plane tetrahedra whose $(1 2  3)$ -faces have Cartan's
invariants with opposite signs. This implies that the map $\pi' :
\mathcal M \longrightarrow \mathcal X'$ is not injective.

For those readers, who prefer to use the standard position of
points, see, for instance, Falbel \cite{F}, here is an explicit
example which illustrates the above situation. To describe this
example, we use horospherical coordinates $(z,v),  \ z \in \mathbb
C,  \ v  \in \mathbb R$, on $\partial\ch{2}$, see Section 1.

We define the following quadruples
$$p(t)=(p_1,p_2,p_3,p_4)=((0,0),\infty, (0,1),(0,t)),$$
and
$$p^{*}(t)=(p_1^*,p_2^*,p_3^*,p_4^*)=((0,0),\infty,
(0,-1),(0,-t)),$$ where $t>0, t\neq 1$.

It is easy to see that the tetrahedra $p(t)$ and $p^{*}(t)$ are
$\mathbb{C}$-plane: they lie in the vertical line (in horospherical
coordinates, this line represents a $\mathbb{C}$-circle, see
\cite{Go}).  The standard lifts for $p_i$ and $p^{*}_i$, see Section
1, are given by the following vectors

$$P_1 \ = \ \begin{bmatrix} 0 \\ 0 \\  1 \end{bmatrix},
\ \ \ \  P_2 \ = \ \begin{bmatrix} 1 \\  0 \\ 0 \end{bmatrix}, \ \ \
P_3 \ = \ \begin{bmatrix} i \\ 0 \\ 1 \end{bmatrix}, \ \ \  P_4 \ =
\
\begin{bmatrix} it \\ 0 \\ 1 \end{bmatrix}, $$

\vspace{6mm}

$$ \ \ \ \  \ P^*_1 \ = \ \begin{bmatrix} 0 \\ 0 \\  1 \end{bmatrix},
\ \ \ \  P^*_2 \ = \ \begin{bmatrix} 1 \\  0 \\ 0 \end{bmatrix}, \ \
\ P^*_3 \ = \ \begin{bmatrix} -i \\ 0 \\ 1 \end{bmatrix}, \ \ \
P^*_4 \ = \   \begin{bmatrix} -it \\ 0 \\ 1 \end{bmatrix}, $$

\vspace{2mm}

\noindent We compute the Hermitian products
$$\herm{P_1,P_2}=1, \ \herm{P_1,P_3}=-i, \ \herm{P_1,P_4}=-it, \
\herm{P_2,P_3}=1, \ \herm{P_2,P_4}=1, \ \herm{P_3,P_4}=i(1-t),$$
$$ \herm{P^*_1,P^*_2}=1, \ \herm{P^*_1,P^*_3}=i, \
\herm{P^*_1,P^*_4}=it, \ \herm{P^*_2,P^*_3}=1, \
\herm{P^*_2,P^*_4}=1, \ \herm{P^*_3,P^*_4}=i(t-1).$$

\noindent Then by definition of the cross-ratios, we have
$$X_1(p(t) )= X_1(p^{*}(t)) = 1/t,  \  X_2 (p(t))= X_2(p^{*}(t)) =
(t-1)/t, \ X_3(p(t) )= X_3(p^{*}(t)) = 1-t.$$

\noindent Also, we compute Cartan's invariants
$$\mathbb{A}(p_1,p_2,p_3)= -\pi/2, \ \ \ \mathbb
{A}(p^*_1,p^*_2,p^*_3)=\pi/2.$$

\vspace{2mm}

Thus, the tetrahedra $p(t)$ and $p^{*}(t)$ have the same
Parker-Platis coordinates. Moreover, it is clear that $p(t)$ and
$p^{*}(t)$ are not congruent with respect to the diagonal action of
${\rm PU}(2,1)$ since
$$\mathbb {A}(p_1,p_2,p_3)=-\mathbb {A
}(p_1^*,p_2^*,p_3^*).$$

\vspace{2mm}

\noindent Let $m(t)$ and $m^{*}(t)$ be the points of $\mathcal M$
corresponding to $p(t)$ and $p^{*}(t)$. We see that $m(t)\neq
m^{*}(t),$  but $\pi' (m(t))=\pi' (m^{*}(t))$. So, we have another
proof of the fact that the map $\pi':\mathcal M \longrightarrow
\mathcal X'$ is not injective.

\vspace{2mm}

\noindent {\bf Remark} \ In Parker-Platis \cite{PP1}, the proof of
the fact that the map $\pi':\mathcal M \longrightarrow \mathcal X'$
is injective is based on Proposition 5.10. Our example shows that
this proposition is not correct.

\vspace{2mm}

\noindent {\bf Remark} \ It is interesting to note that $p(t)$ and
$p^{*}(t)$ are congruent with respect the diagonal action of the
whole isometry group of complex hyperbolic space: there exists an
anti-holomorphic isometry (for instance, the reflection in the real
axes) which sends $p(t)$ to $p^{*}(t)$.

\subsection{The moduli space for the configuration space of ordered quadruples in the
boundary of complex hyperbolic $n$-space}

Let $\mathcal{M}(4,n) $ be the configuration space of ordered
quadruples of distinct points in the boundary of complex hyperbolic
$n$-space, that is, the quotient of the set of ordered quadruples of
distinct points of $\partial\ch{n}$ with respect to the diagonal
action of ${\rm PU}(n,1)$ equipped with the quotient topology.

\vspace{2mm}

In this section, we construct a moduli space for $\mathcal{M}(4,n)$.

\begin{theorem}
$\mathcal {M}(n,4)$  is homeomorphic to the set of points  $X=(X_1,
X_2, A) \in \mathbb C_{*}^{2} \times \mathbb R$ defined by
\begin{equation*}
\dps -2{\rm Re}(X_1+X_2) -2{\rm Re} (X_1\ \overline{X}_2\ e^{-\ i2
A})+ \mod {X_1}^2 +\mod {X_2}^2 +1\leq 0,
\end{equation*}
$$\dps -\pi /2 \leq  A \leq \pi /2,\ \ \ {\rm Re} (X_1 \ e^{-iA})\geq
0.$$ The equality in the first inequality happens if and only if the
quadruples are in the boundary of a complex hyperbolic $2$-space.
\end{theorem}

\begin{proof}
The proof of this theorem is a slight modification of the proof of
Theorem 3.1. The only thing we need is the following proposition
which substitutes Proposition 3.1.
\end{proof}

\begin{prop}
Let $G=(g_{ij})$ be an Hermitian $4\times 4$-matrix such that
$g_{ii}=0$, $g_{12}=g_{23}=g_{34}=1$, $\mod {g_{13}}=1,$ $g_{14}
\neq 0,$ $g_{24} \neq 0.$ Then $G$ is the normalized Gram matrix for
some ordered quadruple $p=(p_1,p_2,p_3,p_4)$  of distinct points of
$\partial\ch{n}$ if and only if {\rm Re}$(g_{13})\leq 0$, {\rm
Re}$(g_{24} \bar{g}_{14})\leq 0$, and $\det G\leq0.$ The determinant
$\det G=0$ if and only if $p$ is in the boundary of a complex
hyperbolic $2$-space.
\end{prop}

\begin{proof} Let us assume that $G$ is the normalized Gram matrix associated
to an ordered quadruple $p=(p_1,p_2,p_3,p_4)$ of distinct points of
$\partial\ch{n}.$ Then it follows from Proposition 2.4 that
$$ \  \ \ \det G(1,2,3)=2 {\rm Re}(\bar{g}_{13}),\ \ \ \ \   \
\ \det G(1,2,4)=2{\rm Re} (g_{24} \bar{g}_{14}).$$ Since the
Hermitian form in the definition of complex hyperbolic $n$-space has
signature $(n,1)$, it follows from Sylvester's Criterion that these
determinants are negative or vanish. Moreover, this also implies
that $\det G \leq 0$. If $p$ is in the boundary of a complex
hyperbolic $2$-space, then $\det G=0$, since in this case any
vectors $P_1,P_2,P_3,P_4$ representing $p_1,p_2,p_3,p_4$ are
linearly dependent.

Now let $G=(g_{ij})$ be an Hermitian $4\times 4$-matrix such that
$g_{ii}=0$, $g_{12}=g_{23}=g_{34}=1$, $\mod {g_{13}}=1,$ ${\rm
Re}{\mit (g}_{13})\leq 0$, ${\rm Re} (g_{24} \bar{g}_{14})\leq 0$,
$\det G \leq 0$. We are going to show that there exist four null
(isotropic) vectors $P_1,P_2,P_3,P_4,$ $P_i \in \mathbb {C}^{n,1},$
whose Gram matrix is equal to $G$.

\vspace{2mm}

We will look for these vectors in the following form:

$$P_1 \ = \ \begin{bmatrix} 0 \\ 0 \\ \vdots \\  1 \end{bmatrix},
\ \ \ \  P_2 \ = \ \begin{bmatrix} 1 \\ 0 \\ \vdots \\ 0
\end{bmatrix}, \ \ \ P_3 \ = \ \begin{bmatrix} z_1 \\ \vdots \\ z_n \\ 1
\end{bmatrix}, \ \ \ P_4 \ = \   \begin{bmatrix} w_1 \\ \vdots \\ w_n \\
w_{n+1}
\end{bmatrix},
$$ where $z_i$, $w_i$ are complex numbers,  and $\mod{z_1}=1.$

\vspace{2mm}

Then we have
 $$\herm{P_1,P_2}=1, \ \herm{P_1,P_3}=\bar{z}_{1}, \
\herm{P_1,P_4}=\bar{w}_{1}, \ \herm{P_2,P_3}=1, \
\herm{P_2,P_4}=\bar{w}_{n+1}, $$
$$ \herm{P_3,P_4}=z_1 \bar{w}_{n+1} \
+ z_2 \bar{w}_2 \  + \cdots +  \ z_n \bar{w}_n  \ +\bar{w}_1 \ = \
z_1 \bar{w}_{n+1} \ + \bar{w}_1 \ + \hherm{z,w},
$$
where $z = (z_2, \ldots ,z_n),$  $w=(w_2, \ldots , w_n)$, and
$\hherm{z,w}$ is the standard Hermitian product on $\mathbb
C^{n-1}.$

\vspace{2mm}

Since we need $P_3$ and $P_4$ to be null vectors, we have the
following equations:
$$ z_1 + \mod{z}^2 + \bar{z}_1 =0, \ \ \ \ \ \ \ w_1 \bar{w}_{n+1} \
+ \mod{w}^2 + \bar{w}_1 w_{n+1} =0,
$$
where $\mod{z}^2=\hherm{z,z},$ and $\mod{w}^2=\hherm{w,w}.$

\vspace{2mm}

From the definition of the Gram matrix, we have
$$
g_{12}=1,\ \  g_{13}=\bar{z}_1 ,\ \  g_{14}=\bar{w}_1 , \ \
g_{23}=1, \ \  g_{24}= \bar{w}_{n+1} ,
$$
and
$$
g_{34}= z_1 \bar{w}_{n+1} +  \hherm{z,w} + \bar{w}_1 = 1.
$$
This implies that we have already found $z_1 , w_1, w_{n+1} $ in
terms of $g_{ij}$. Therefore, we need to find the vectors $z$ and
$w$.

\vspace{2mm}

We consider the following system of equations
$$ (1)\ \ z_1 + \mod{z}^2 + \bar{z}_1 =0, \ \ \ \ \
(2) \  \  w_1 \bar{w}_{n+1} + \mod{w}^2 + \bar{w}_1 w_{n+1} =0, \ \
\ \ \ (3) \ \  z_1 \bar{w}_{n+1} + \hherm{z,w} + \bar{w}_1 = 1,
$$ and show that it has a solution under the conditions of the
proposition. We write the first two equations in the following form:
$$\mod{z}^2 = -2 {\rm Re} ({z}_1) = -2{\rm Re}(g_{13}), \ \ \ \ \
\mod{w}^2 = -2{\rm Re}(w_1 \bar{w}_{n+1}) = -2{\rm Re}(g_{24}
\bar{g}_{14}).$$ We immediately see that there exist solutions for
$\mod{z}$ and $\mod{w}$ under our conditions. The third equation can
be written as
$$  \hherm{z,w} = 1 \ - \  z_1 \bar{w}_{n+1} \  -  \ \bar{w}_1.$$

\vspace{2mm}

\noindent Let us write the Cauchy-Schwarz inequality for the vectors
$z$ and $w$:
$$\mod {\hherm{z,w}}^2 \leq \mod{z}^2 \ \mod{w}^2.$$
Substituting $\mod{z}$ and $\mod{w}$ from equations (1) and (2) and
$\hherm{z,w}$ from the third equation, we rewrite this inequality in
the following form:
$$\mod {1 \ - \  z_1 \bar{w}_{n+1} \  -  \ \bar{w}_1}^2  -(z_1 \ + \ \bar{z}_1)
\ (w_1 \bar{w}_{n+1} \ +  \ \bar{w}_1 w_{n+1})\leq 0.$$
By computations similar to those in Proposition 3.1, we have
 $$\mod {1 \ - \
z_1 \bar{w}_{n+1} \  -  \ \bar{w}_1}^2  -(z_1 \ + \ \bar{z}_1) \
(w_1 \bar{w}_{n+1} \ +  \ \bar{w}_1 w_{n+1}) = \ \det G.$$ This
implies that if there exist solutions $(z,w)$ to the equation
$$\mod { \hherm{z,w}} = \mod {1 \ - \  z_1 \bar{w}_{n+1} \  -  \ \bar{w}_1}$$
that satisfy equations $(1)$ and $(2)$ then necessarily $\det G \leq
0.$ Let us assume now that $\det G \leq 0.$ Then there exist vectors
$z$ and $w$ satisfying equations $(1)$ and $(2)$ such that the
inequality
$$\mod{1 \ - \  z_1 \bar{w}_{n+1} \  -  \ \bar{w}_1}^2 \  \leq\  \mod{z}^2
\ \mod{w}^2
$$
holds. One verifies that if one of the vectors, $z$ or $w$, is the
null vector, then the system above has a solution because the third
equation is satisfied automatically in this case provided that $\det
G \leq 0$. So, we may suppose that both vectors $z$ and $w$ are not
null. Let $\mathbb{C} z $ and $\mathbb{C} w$ be the complex lines in
the underlying real vector space of $\mathbb C^{n-1}$ spanned by $z$
and $w$, and let $\angle(\mathbb{C} z,\mathbb{C} w)$ be the angle
between $\mathbb{C} z $ and $\mathbb{C} w$, see \cite{Go}. Then it
follows from the formula
$$ \mod{\hherm{z,w}} = \mod{z} \mod{w} \cos (\angle(\mathbb{C} z,\mathbb{C}
w))$$ proved in Lemma 2.2.2, \cite{Go}, that by choosing an
appropriate angle between $\mathbb{C} z $ and $\mathbb{C} w$
(without changing the norms of $z$ and $w$), we may assume that $z$
and $w$ satisfy the equality $\mod { \hherm{z,w}} = \mod {1 -  z_1
\bar{w}_{n+1} - \bar{w}_1}.$  Let $z'= e^{i \theta}z$. Then
$\hherm{z',w} = \ e^{i \theta} \ \hherm{z,w}.$ This implies that
there exists $\theta$ such that $(z',w)$ is a solution to the
equation $ \hherm{z,w} = 1 - z_1 \bar{w}_{n+1}  - \bar{w}_1$ (here
we have used the fact that if two complex numbers have the same norm
then there exists a rotation which sends one number to another).
Finally, it is easy to see that if $g_{14} \neq 0,$ $g_{24} \neq 0,$
then the points $p_i$ defined by the vectors $P_i$ are distinct.
This proves the statement of the proposition.
\end{proof}



\bigskip

$\begin{array}{ll}
\textrm{\bf E-mail\ addresses:}  &  \textrm{cunha@mat.ufmg.br} \\
                                 &  \textrm{nikolay@mat.ufmg.br}
\end{array}$


\begin{thebibliography}{88888}

\bibitem{Br} U. Brehm,  The shape invariant of triangles and trigonometry in
two-point homogeneous spaces. Geom. Dedicata 33 (1990), no. 1,
59--76.

\bibitem{BrEt} U. Brehm, B. Et-Taoui, Congruence criteria for finite
subsets of complex projective and complex hyperbolic spaces.
Manuscripta Math. 96 (1998), no. 1, 81--95.

\bibitem{BS} D. Burns, S. Shnider,  Spherical hypersurfaces in complex
manifolds. Invent. Math. 33 (1976), no. 3, 223--246.

\bibitem{C} E. Cartan, Sur le groupe de la g\'{e}om\'{e}trie hypersph\'{e}rique. Comm.
Math. Helv. 4 (1932), 158--171.

\bibitem{DG} F. Dutenhefner, N. Gusevskii, Complex hyperbolic Kleinian groups
with limit set a wild knot. Topology 43 (2004), 677--696.

\bibitem{E}  D. Epstein,  Complex hyperbolic geometry.
Analytical and geometric aspects of hyperbolic space
(Coventry/Durham, 1984), 93--111, London Math. Soc. Lecture Note
Ser., 111, Cambridge Univ. Press, Cambridge, 1987.

\bibitem{F} E. Falbel, Geometric structures associated to triangulations as fixed point sets of involutions.
Topology Appl. 154 (2007), no. 6, 1041--1052.

\bibitem{FP} E. Falbel, I. Platis, The $\rm PU(2,1)$ configuration space of four points in $S\sp 3$ and the cross-ratio variety.
 Math. Ann. 340 (2008), no. 4, 935--962.

\bibitem{Go} W.M. Goldman, Complex hyperbolic geometry. Oxford Mathematical
Monographs. Oxford Science Publications. The Clarendon Press, Oxford
University Press, New York, 1999. xx+316 pp.

\bibitem{GoP}  W.M. Goldman, J.R. Parker, Dirichlet polyhedra for
dihedral groups acting on complex hyperbolic space. J. Geom. Anal. 2
(1992), no. 6, 517--554.
\bibitem{Gr} C. Grossi, PhD Thesis, Universidade Estadual de
Campinas, 2006.

\bibitem{GuP1}  N. Gusevskii, J.R. Parker,  Representations of free Fuchsian
groups in complex hyperbolic space. Topology 39 (2000), no. 1,
33--60.

\bibitem{GuP2} N. Gusevskii, J.R. Parker, Complex hyperbolic
quasi-Fuchsian groups and Toledo's invariant.
 Special volume dedicated to the memory of Hanna Miriam Sandler (1960--1999). Geom. Dedicata 97 (2003), 151--185.

\bibitem{HS} J. Hakim, H. Sandler, The moduli space of $n+1$ Points
in Complex Hyperbolic $n$-Space. Geom. Dedicata 97 (2003), 3-15.


\bibitem{KR} A. Kor\'anyi, H. M. Reimann, The complex cross ration
on the Heisenberg group. Enseign. Math. 33 (1987), no.(3-5),
291-300.

\bibitem{PP1} J.R. Parker, I. Platis, Complex hyperbolic
Fenchel-Nielsen coordinates. Topology 47 (2008), no. 2, 101--135.

\bibitem{PP2} J. R. Parker, I.  Platis. Global, geometrical coordinates on Falbel's cross-ratio variety. Canad. Math.
Bull., to appear.

\bibitem{Sch} W. Scharlau, Quadratic and Hermitian forms. Grundlehren der
Mathematischen Wissenschaften [Fundamental Principles of
Mathematical Sciences], 270. Springer-Verlag, Berlin, 1985. x+421
pp.

\bibitem{W} P. Will, Traces, Cross-ratios and 2-generator Subgroups of PU(2,1). Canad. J. Math., to
appear.



\end{thebibliography}
\end{document}